\documentclass[a4paper,11pt]{article}

\usepackage[utf8]{inputenc}
\usepackage{lmodern}
\usepackage[english]{babel}
\usepackage{amsmath}
\usepackage{amsfonts}
\usepackage{hyperref}
\usepackage{dsfont}			
\usepackage{amsmath, amsthm, amssymb, amsthm,bbm,bm,mathrsfs}

\newtheorem{prop}{Proposition}[section]
\newtheorem{theorem}[prop]{Theorem}
\newtheorem{corollary}[prop]{Corollary}
\newtheorem{lemma}[prop]{Lemma}

\newcommand{\diff}{\mathrm{d}}
\newcommand{\id}{\mathrm{id}}

\title{A quantitative fourth moment theorem in free probability theory}
\author{Guillaume C\'ebron \thanks{Guillaume C\'ebron was partially supported by the ERC advanced grant “non-commutative distributions in free
probability”, held by Roland Speicher.}\\IMT; UMR 5219 \\
Université de Toulouse; CNRS \\
UPS, F-31400 Toulouse, France \\
\texttt{guillaume.cebron@math.univ-toulouse.fr}}
\begin{document}

\maketitle

   \begin{abstract}A quantitative "fourth moment theorem" for any self-adjoint element in a homogeneous Wigner chaos is provided: the Wasserstein distance is controlled by the distance from the fourth moment to two. The proof uses the free counterpart of the Stein discrepancy. On the way, the free analogue of the WSH inequality is established.
      \end{abstract}

\section{Introduction}
\subsection{A quantitative fourth moment theorem}
Let $\mathcal{S}$ be a family of centered jointly semicircular variables in some tracial $W^*$-probability space $(\mathcal{A},\tau)$ such that the Hilbert space in $L^2(\mathcal{A},\tau)$ generated by $\mathcal{S}$ is separable. For any integer $n\geq 0$, let us denote by $\mathcal{P}_n$ the \emph{Wigner chaos} of order $n$, that is to say the (complex) Hilbert space in $L^2(\mathcal{A},\tau)$ generated by the set
$\{1_{\mathcal{A}}\}\cup\{S_1\cdots S_k:1\leq k\leq n,S_1,\ldots, S_k \in \mathcal{S}\}$. Let us denote by $\mathcal{H}_n$ the \emph{homogeneous Wigner chaos} of order $n$, defined by
$$\mathcal{H}_n:=\mathcal{P}_n\ominus \mathcal{P}_{n-1}\ \ (=\mathcal{P}_n\cap \mathcal{P}_{n-1}^{\perp}).$$
In this article, the following quantitative bound of the $2$-Wasserstein distance $W_2$ from an element of the  $n$-th homogeneous chaos to the semicircular variable is provided.

\begin{theorem}[see Theorem~\ref{mainth}]\label{introth}Let $n\geq 2$. Let us consider a self-adjoint element $F=F^*$ of the $n$-th homogeneous chaos $\mathcal{H}_n$ such that $\tau(F^2)=1$, and let $S$ be a standard semicircular variable. We have $$W_2(F,S) \leq  n^{3/4}\left(\tau(F^4)-2\right)^{1/4}.$$
\end{theorem}
In \cite{KNPS2012}, Kemp, Nourdin, Peccati, and Speicher showed that, for a sequence of such variables in a fixed homogeneous chaos, convergence of the fourth moment controls convergence in distribution towards the semicircular law. However, they provided a quantitative bound in terms of the fourth moment only in the special case of Wigner chaos of order two: they proved that, for $F$ in the $2$-th homogeneous chaos,
$$d_{\mathcal{C}_2}(F,S)  \leq \frac{1}{2}\sqrt{\frac{3}{4}}\left(\tau(F^4)-2\right)^{1/2}$$
for an ad hoc distance $d_{\mathcal{C}_2}$ which is in fact weaker than the Wassertein distance $W_2$ (see Lemma~\ref{prop:W}). Whether a similar fourth moment bound holds for chaoses of higher orders, as in the commutative setting (see Nualart and Peccati~\cite{NP2005} and Nourdin and Peccati~\cite{NP2009}), is a question which has first been investigated by Bourguin and Campese in~\cite{BC2017}. They provided the following bound
$$d_{\mathcal{C}_2}(F,S)  \leq \sqrt{C_n}\left(\tau(F^4)-2\right)^{1/2}$$
for variables in a homogeneous Wigner chaos of order $n$ which are \emph{fully symmetric} (see Section~\ref{Fully}), and for a constant $C_n$ which grows asymptotically linearly with $n$.
Theorem~\ref{introth} completes their result: the assumption of full symmetry is removed, and the distance $d_{\mathcal{C}_2}$ is replaced by the Wasserstein distance $W_2$. The proof of Theorem~\ref{introth} uses the notions of free Stein kernel and free Stein discrepancy introduced by Fathi and Nelson in~\cite{FN2016} (see Section~\ref{Sec:Stein}), and the free Malliavin calculus as defined by Biane and Speicher in~\cite{BS1998} (see Section~\ref{Sec:Wigner}).
One can consult \cite{BC2017} and the reference therein to get a larger overview about the classical and free results related to the "fourth moment theorem".


\subsection{Notation}A \emph{non-commutative probability space} $(\mathcal{A},\tau)$ is a unital $*$-algebra with a linear functional $\tau:\mathcal{A}\to \mathbb{C}$ such that $\tau(1_{\mathcal{A}})=1$ and $\tau(A^*A)\geq 0$ for all $A\in \mathcal{A}$. Let $\left(\mathcal{A},\tau\right)$ be a \emph{tracial $W^*$-probability space}, that is to say a non-commutative probability space such that $\mathcal{A}$ is a von Neumann algebra, and $\tau$ is a faithful normal tracial state. For all $A_1,\ldots,A_n\in \mathcal{A}$,  we denote by $W^*(A_1,\ldots,A_n)$ the von Neumann subalgebra of $\mathcal{A}$ generated by $A_1,\ldots,A_n$.

Let $\mathcal{A}_1$ and $\mathcal{A}_2$ be two von Neumann subalgebras of $\mathcal{A}$. These algebras are called \emph{free} if, for all $n\in \mathbb{N}$, and all indices $i_1 \neq i_2 \neq \ldots \neq i_n$, whenever $A_j \in \mathcal{A}_{i_j}$ and $\tau(A_j) = 0$ for all $1\leq j \leq n$, we have $\tau(A_1\cdots A_n) = 0$. We say that a family $A_1,\ldots, A_n \in \mathcal{A}$ is \emph{free} from another family $B_1,\ldots, B_m \in \mathcal{A}$ if the algebras $W^*(A_1,\ldots, A_n)$ and $W^*(B_1,\ldots, B_n)$ are free.

In this article, $H_1\otimes H_2$ denotes the \emph{algebraic tensor product} of $H_1$ and $H_2$ if $H_1$ and $H_2$ are only vector spaces, and $H_1\otimes H_2$ denotes the \emph{Hilbert space tensor product} of $H_1$ and $H_2$ whenever $H_1$ and $H_2$ are two Hilbert spaces.

We denote by $L^2(\mathcal{A},\tau)$ the Hilbert space given by the completion of $\mathcal{A}$ with respect to the inner product $\langle A,B\rangle_{L^2(\mathcal{A},\tau)}=\tau(A^*B)$ (remark that if $\mathcal{A}$ is a subspace of bounded operators, $L^2(\mathcal{A},\tau)$ can be identified with a space of operators affiliated with $\mathcal{A}$). If $\mathcal{B}$ is a $W^*$-subalgebra of $\mathcal{A}$, we denote by $L^2(\mathcal{B},\tau)$ the Hilbert space generated by $\mathcal{B}$ in $L^2(\mathcal{A},\tau)$. Moreover, for all $A \in \mathcal{A}$, we denote by $\tau(B|A)$ the orthogonal projection of an element $B\in L^2(\mathcal{A},\tau)$ onto $L^2(W^*(A),\tau)$ and call it the \emph{conditional expectation} of $B$ given $A$.

Finally, given a self-adjoint element $A=A^*$ of $\mathcal{A}$, there exists a unique probability measure $\mu$ on $\mathbb{R}$, called the \emph{distribution} of $A$, such that, for all bounded Borel function $h$ on $\mathbb{R}$, we have
$$\tau(h(A))=\int_\mathbb{R}h\diff \mu.$$

\subsection{Organisation of the paper}

In Section~\ref{Sec:Stein}, the free entropy, the free Stein kernel and the free Stein discrepancy are introduced. Proposition~\ref{propWS} states that the Wasserstein distance is bounded by the Stein discrepancy, and Proposition~\ref{freeWSH} is the free analogue of the WSH inequality, that is to say an improvement of the free Talagrand inequality involving the Wasserstein distance, the Stein discrepancy and the free entropy.

In Section~\ref{Sec:Wigner}, the Wigner chaoses and the Malliavin derivative are introduced. Corollary~\ref{corStein} exhibits a Stein kernel for any element of the Wigner chaos, and Proposition~\ref{propFourthmoment} states that the free Stein discrepancy of an element in a homogeneous chaos is controlled by the distance from the fourth moment to $2$. Finally, our quantitative "fourth moment theorem" is summed up in Theorem~\ref{mainth}.

\section{Functional inequalities involving the Stein discrepancy}\label{Sec:Stein}
The \emph{noncommutative derivative}, or \emph{free difference quotient}, is the linear map $\partial$ from the set of polynomial $\mathbb{C}[X]$ to its tensor product $\mathbb{C}[X]\otimes \mathbb{C}[X]$ defined by $\partial X^n=\sum_{k=1}^n X^{k-1}\otimes X^{n-k}.$ A self-adjoint element $S$ of $(\mathcal{A},\tau)$ is called a \emph{standard semicircular variable} if the distribution of $S$ has density $\sqrt{4-t^2}/2\pi$ on $[-2,2]$. Equivalently, a self-adjoint element $S$ of $(\mathcal{A},\tau)$ is a standard semicircular variable if it satisfies the following \emph{Schwinger-Dyson equation}: for all polynomials $P\in \mathbb{C}[X]$, we have
$$\langle S,P(S) \rangle_{L^2(\mathcal{A},\tau)}=\langle 1_{\mathcal{A}}\otimes 1_{\mathcal{A}},\partial P(S) \rangle_{L^2(\mathcal{A},\tau)\otimes L^2(\mathcal{A},\tau)}.$$
In the two following sections, we present two different ways of deforming this equation. One leads to the definition of the free Fisher information, due to Voiculescu \cite{V1998}. The other leads to the definition of the free Stein discrepancy, due to Fathi and Nelson \cite{FN2016}.
\subsection{Free Fisher information and free entropy}
Let $F$ be a self-adjoint element $F$ of $(\mathcal{A},\tau)$. The \emph{conjugate variable} of $F$, if it exists, is the element $\xi\in L^2(W^*(F),\tau)$ such that, for all $P\in \mathbb{C}[X]$,
$$\langle \xi,P(F) \rangle_{L^2(\mathcal{A},\tau)}=\langle 1_{\mathcal{A}}\otimes 1_{\mathcal{A}},\partial P(F) \rangle_{L^2(\mathcal{A},\tau)\otimes L^2(\mathcal{A},\tau)}.$$
The \emph{free Fisher information} of $F$ relative to $S$ is the constant
$$\Phi^*(F|S):=\|\xi-F\|^2_{L^2(\mathcal{A},\tau)}$$
if there exists a conjugate variable of $F$, and $+\infty$ if not. Even if $\Phi^*(F|S)=+\infty$, it is possible to perturb $F$ in order to get a finite free Fisher information. Without loss of generality, let us assume that there exists a standard semicircular variable $S$ such that $F$ and $S$ are free in $(\mathcal{A},\tau)$, and set
$F_t:=e^{-t}F+\sqrt{1-e^{-2t}}S.$
For all $t>0$, 
the free Fisher information of $F_t$ is finite. The \emph{free entropy} of $F$ relative to $S$ is the quantity
$$\chi^*(F|S):=\int_0^\infty \Phi^*(F_t|S)\diff t\ \ \in [0,+\infty].$$
Denoting by $\mu$ the distribution of $F$, we have also the explicit formula
$$\chi^*(F|S)=\frac{1}{2}\int_\mathbb{R}x^2\diff \mu(x)-\iint_{\mathbb{R}^2}\log |x-y|\diff \mu(x)\diff \mu(x)-\frac{3}{4}.$$
By definition, the Wasserstein distance $W_2(A,B)$ from $A$ to $B$ is given by
$$W_2(A,B)^2=\inf_{\pi\in \Pi(A,B)}\int_{\mathbb{R}^2}|x-y|^2 d\pi(x,y),$$
where $\Pi(A,B)$ denotes the set of probability measures on $\mathbb{R}^2$ with marginals the distributions of $A$ and of $B$.
The free Fisher information and the free entropy dominate the Wasserstein distance of the distribution of $F$ to a semicircular variable $S$ thanks to the following free analogues of Talagrand's inequality and log-Sobolev inequality.
\begin{theorem}[Theorem 2.8 of \cite{BV2001} and Section 7.2. of \cite{BS2001}]For all self-adjoint element $F$ of $(\mathcal{A},\tau)$, we have
$$W_2(F,S)^2\leq 2\ \chi^*(F|S)\leq \Phi^*(F|S).$$
\end{theorem}
Note that the proof of the first inequality above uses the following lemma, where $\frac{\diff^+}{\diff t}$ stands for the upper right derivative.
\begin{lemma}[Lemma 2.7 of \cite{BV2001}]\label{lemmaWI}For all self-adjoint element $F$ of $(\mathcal{A},\tau)$, we have
$$\frac{\diff^+}{\diff t}W_2(F,F_t)\leq (\Phi^*(F_t|S))^{1/2}.$$
\end{lemma}
Finally, let us state the free version of the de Bruijn’s formula.
\begin{lemma}[Lemma 2.1 of \cite{FN2016}]\label{deBruijn}For all self-adjoint element $F$ of $(\mathcal{A},\tau)$, we have
$$\frac{\diff}{\diff t}\chi^*(F_t|S)= -\Phi^*(F_t|S).$$
\end{lemma}
\subsection{Free Stein discrepancy}
A \emph{free Stein kernel} of $F$, if it exists, is an element $A\in L^2(\mathcal{A},\tau)\otimes L^2(\mathcal{A},\tau)$ such that, for all $P\in \mathbb{C}[X]$,
$$\langle F,P(F) \rangle_{L^2(\mathcal{A},\tau)}=\langle A,\partial P(F) \rangle_{L^2(\mathcal{A},\tau)\otimes L^2(\mathcal{A},\tau)}.$$
The \emph{free Stein discrepancy} $\Sigma^*(F|S)$ of $F$ relative to $S$ is defined as the constant
$$\Sigma^*(F|S)=\inf_A\|A-1_{\mathcal{A}}\otimes 1_{\mathcal{A}}\|_{L^2(\mathcal{A},\tau)\otimes L^2(\mathcal{A},\tau)},$$
where the infimum is taken over all the free Stein kernels $A$ of $F$. We have the following free analogue of the HSI inequality, which is an improvement of the free log-Sobolev inequality $2\ \chi^*(F|S)\leq \Phi^*(F|S)$, because $\log(1+x)\leq x$.
\begin{theorem}[Theorem 2.6 of \cite{FN2016}]\label{thHSI}For all self-adjoint element $F$ of $(\mathcal{A},\tau)$, we have
$$2\ \chi^*(F|S)\leq \Sigma^*(F|S)^2 \log\left(1+\frac{\Phi^*(F|S)}{\Sigma^*(F|S)^2}\right).$$
\end{theorem}
Note that, if $A$ is a free Stein kernel of $F$, the orthogonal projection of $A$ into $L^2(W^*(F),\tau)\otimes L^2(W^*(F),\tau)$ is also a free Stein kernel of $F$. Furthermore, if $F$ is free from a self-adjoint element $G\in \mathcal{A}$, and $A\in L^2(W^*(F),\tau)\otimes L^2(W^*(F),\tau)$ is a free Stein kernel of $F$, the intertwining relation between $\partial$ and the conditional expectation found by Voiculescu in \cite{V2000} (see also \cite[Lemma 20]{MS2017}) gives the following fact: for all $P\in \mathbb{C}[X]$,
$$\langle F,P(F+G) \rangle_{L^2(\mathcal{A},\tau)}=\langle A,\partial P(F+G) \rangle_{L^2(\mathcal{A},\tau)\otimes L^2(\mathcal{A},\tau)}.$$
As a consequence, we get the following decay of the Stein discrepancy.
\begin{lemma}\label{lemmaSigma}Let $F$ be a self-adjoint element of $(\mathcal{A},\tau)$, $S$ be a standard semicircular variable free from $F$ and, for all $t\geq 0$, set $F_t:=e^{-t}F+\sqrt{1-e^{-2t}}S$.
Then, for all $t\geq 0$, the element $$e^{-2t}A+(1-e^{-2t})\cdot 1_{\mathcal{A}}\otimes 1_{\mathcal{A}}$$ is a free Stein kernel of $F_t$ whenever $A\in L^2(W^*(F),\tau)\otimes L^2(W^*(F),\tau)$ is a free Stein kernel of $F$. Moreover,
$$\Sigma^*(F_t|S)\leq e^{-2t} \Sigma^*(F|S).$$
\end{lemma}

\subsection{Wasserstein distance and free Stein discrepancy}
We have the following inequality, which is an extension of  \cite[Lemma 2.4]{FN2016} in the semicircular case.
\begin{lemma}\label{lemmaIS}Let $F$ be a self-adjoint element of $(\mathcal{A},\tau)$ and let $S$ be a standard semicircular variable free from $F$. For all $t\geq 0$, setting $F_t=e^{-t}F+\sqrt{1-e^{-2t}}S$, we have
$$\Phi^*(F_t|S)^{1/2}\leq \frac{e^{-2t}}{\sqrt{1-e^{-2t}}}\Sigma^*(F|S) .$$
\end{lemma}
\begin{proof}We can assume that $\Sigma^*(F|S)<+\infty$. Let us denote by $A$ a Stein kernel of $F$ and, because the orthogonal projection of a Stein kernel into $L^2(W^*(F),\tau)\otimes L^2(W^*(F),\tau)$ is a Stein kernel, we assume without loss of generality  that $A\in L^2(W^*(F),\tau)\otimes L^2(W^*(F),\tau)$. Because $F$ is free from $S$, we can consider the map $$\partial_S:L^2(W^*(F,S),\tau)\to L^2(W^*(F,S),\tau)\otimes L^2(W^*(F,S),\tau)$$ (respectively its adjoint $\partial_S^*$) defined in Section $3$ of \cite{V1998} as a closed unbounded operator whose domain contains $\mathbb{C}\langle F,S\rangle$ (respectively $\mathbb{C}\langle F,S\rangle^{\otimes 2}$) . The map $\partial_S$ is given, for all monomial $M$ in the variables $F$ and $S$, by
$$\partial_SM=\sum_{M=M_1SM_2}M_1\otimes M_2, $$
where the sum runs over all decompositions of
$M$ in the form $M=M_1SM_2$ with some monomials $M_1$ and $M_2$. By \cite[Proposition 4.3]{V1998} The map $\partial_S^*$ is given, for polynomials $P_1$ and $P_2$ in the variables $F$, by
$$\partial_S^*(P_1\otimes P_2)=P_1SP_2.$$ 
As a consequence, $\partial_S\partial_S^*$ is the identity on $\mathbb{C}\langle F\rangle^{\otimes 2}$, and, for all element $B\in \mathbb{C}\langle F\rangle^{\otimes 2}$, we have
$$\|\partial_S^*B\|_{L^2(\mathcal{A}, \tau)}^2=\langle\partial_S\partial_S^*B,B\rangle_{L^2(\mathcal{A}, \tau)\otimes L^2(\mathcal{A}, \tau)}^2=\|B\|_{L^2(\mathcal{A}, \tau)\otimes L^2(\mathcal{A}, \tau)}^2.$$
It implies that all elements in $L^2(W^*(F),\tau)\otimes L^2(W^*(F),\tau)$ are in the domain of $\partial_S^*$, and that $\partial_S^*$ maps isometrically $L^2(W^*(F),\tau)\otimes L^2(W^*(F),\tau)$ into $L^2(W^*(F,S),\tau)$.

As particular cases, $1_{\mathcal{A}}\otimes 1_{\mathcal{A}}$ and $A$ are in the domain of $\partial_S^*$. We know that $\partial_S^*(1_{\mathcal{A}}\otimes 1_{\mathcal{A}})=S$, but the computation of $\partial_S^*(A)$ is more difficult. However, we can say that under the conditional expectation $\tau(\cdot|F_t)$, we have
$$\sqrt{1-e^{-2t}}\tau(F|F_t)=e^{-t}\tau(\partial_S^*A|F_t).$$
Indeed, for all polynomial $P$, we have
\begin{multline*}
$$\sqrt{1-e^{-2t}}\langle P(F_t),F\rangle_{L^2(\mathcal{A}, \tau)}=\sqrt{1-e^{-2t}}e^{-t}\langle (\partial P)(F_t),A\rangle_{L^2(\mathcal{A}, \tau)}\\=e^{-t}\langle \partial_S (P(F_t)),A\rangle_{L^2(\mathcal{A}, \tau)}=e^{-t}\langle P(F_t),\partial_S^*A\rangle_{L^2(\mathcal{A}, \tau)}.
\end{multline*}
Finally, \cite[Corollary 3.9]{V1998} tells us that
$$\xi_t:=\frac{1}{\sqrt{1-e^{-2t}}}\tau(S|F_t)$$ is the conjugate variable of $F_t$.

In the following computation, the right member of the scalar product is always in $W^*(F_t)$. Thus, we have the freedom of replacing $\xi_t$ by $\frac{1}{\sqrt{1-e^{-2t}}}S$ and $\sqrt{1-e^{-2t}}F$ by $e^{-t}\partial_S^*A$ in the left member of the scalar product because they coincide under the conditional expectation $\tau(\cdot|F_t)$. We can compute
\begin{align*}\Phi^*(F_t|S)=& \left\langle \xi_t-F_t,\xi_t-F_t\right\rangle_{L^2(\mathcal{A}, \tau)}\\
=&\left\langle \xi_t-e^{-t}F-\sqrt{1-e^{-2t}}S,\xi_t-F_t\right\rangle_{L^2(\mathcal{A}, \tau)}\\
=&\frac{1}{\sqrt{1-e^{-2t}}}\left\langle S-\sqrt{1-e^{-2t}}e^{-t}F-(1-e^{-2t})S,\xi_t-F_t\right\rangle_{L^2(\mathcal{A}, \tau)}\\
=&\frac{1}{\sqrt{1-e^{-2t}}}\left\langle e^{-2t}S-e^{-2t}\partial_S^*A,\xi_t-F_t\right\rangle_{L^2(\mathcal{A}, \tau)}\\
=&\frac{e^{-2t}}{\sqrt{1-e^{-2t}}} \left\langle \partial_S^*(1_{\mathcal{A}}\otimes 1_{\mathcal{A}}-A),\xi_t-F_t\right\rangle_{L^2(\mathcal{A}, \tau)}\\
\leq& \frac{e^{-2t}}{\sqrt{1-e^{-2t}}}\| \partial_S^*(1_{\mathcal{A}}\otimes 1_{\mathcal{A}}-A)\|_{L^2(\mathcal{A}, \tau)}\Phi^*(F_t|S)^{1/2}\\
&\hspace{2cm} =\frac{e^{-2t}}{\sqrt{1-e^{-2t}}}\|1_{\mathcal{A}}\otimes 1_{\mathcal{A}}-A\|_{L^2(\mathcal{A}, \tau)\otimes L^2(\mathcal{A}, \tau)}\Phi^*(F_t|S)^{1/2},
\end{align*}
and we conclude by minimizing over $A$.
\end{proof}
We get the following proposition, which is the free counterpart of \cite[Proposition 3.1]{LNP2015}.
\begin{prop}\label{propWS}For all self-adjoint random variable $F$ of $(\mathcal{A},\tau)$, we have
$$ W_2(F,S) \leq \Sigma^*(F|S).$$
\end{prop}
\begin{proof}Let $S$ be a standard semicircular variable free from $F$. For all $t\geq 0$, set $F_t=e^{-t}F+\sqrt{1-e^{-2t}}S.$ Lemma~\ref{lemmaWI} and Lemma~\ref{lemmaIS} tell us that
$$\frac{d^+}{dt}W_2(F,F_t)\leq (\Phi^*(F_t|S))^{1/2}\leq  \frac{e^{-2t}}{\sqrt{1-e^{-2t}}}\Sigma^*(F|S).$$
Finally, we integrate on $[0,+\infty[$ and we get $ W_2(F,S) \leq \Sigma^*(F|S).$
\end{proof}
\subsection{The free WSH inequality}We have also the following free analogue of the WSH inequality (see \cite[Theorem 3.2]{LNP2015}), which is an improvement of the free Talagrand inequality $W_2(F,S)^2\leq 2\ \chi^*(F|S)$, because $\arccos(e^{-x})\leq \sqrt{2x}$.
\begin{prop}\label{freeWSH}For all self-adjoint element $F$ of $(\mathcal{A},\tau)$, we have
$$ W_2(F,S) \leq \Sigma^*(F|S)\arccos\left(e^{-\frac{\chi^*(F|S)}{\Sigma^*(F|S)^2}}\right).$$
\end{prop}
The proof is mutatis mutandis the one of \cite[Theorem 3.2]{LNP2015}.
\begin{proof}Let $S$ be a standard semicircular variable free from $F$. For all $t\geq 0$, set $F_t=e^{-t}F+\sqrt{1-e^{-2t}}S.$ Theorem~\ref{thHSI} applied to $F_t$ gives us
$$\chi^*(F_t|S)\leq \frac{1}{2}\Sigma^*(F_t|S)^2 \log\left(1+\frac{\Phi^*(F_t|S)}{\Sigma^*(F_t|S)^2}\right).$$
Now, $\Sigma^*(F_t|S)\leq \Sigma^*(F|S)$ by Lemma~\ref{lemmaSigma} and $r\mapsto r\log(1+\Phi^*(F_t|S)/r)$ is increasing, from which it follows that
$$\chi^*(F_t|S)\leq \frac{1}{2}\Sigma^*(F|S)^2 \log\left(1+\frac{\Phi^*(F_t|S)}{\Sigma^*(F|S)^2}\right).$$
This inequality is equivalent to
$$\sqrt{\Phi^*(F_t|S)}\leq \frac{\Phi^*(F_t|S)}{\Sigma^*(F|S)\sqrt{e^{\frac{2\chi^*(F_t|S)}{\Sigma^*(F|S)^2}}-1}}.$$
Using Lemma~\ref{lemmaWI} and Lemma~\ref{deBruijn}, we get
$$\frac{d^+}{dt}W_2(F,F_t)\leq \sqrt{\Phi^*(F_t|S)}\leq  \frac{-\frac{\diff}{\diff t}\chi^*(F_t|S)}{\Sigma^*(F|S)\sqrt{e^{\frac{2\chi^*(F_t|S)}{\Sigma^*(F|S)^2}}-1}}=-\frac{\diff}{\diff t} \left(\Sigma^*(F|S)\arccos\left(e^{-\frac{\chi^*(F_t|S)}{\Sigma^*(F|S)^2}}\right)\right).$$
Finally, we integrate on $[0,+\infty[$ and we get the result because $\chi^*(S|S)=0$.
\end{proof}
\subsection{The multivariate case}
All the definitions of the current section make perfectly sense for a $n$-tuple $F=(F_1,\ldots,F_n)$ of self-adjoint elements in $(\mathcal{A}, \tau)$ and a $n$-tuple of free standard semicircular variables $S=(S_1,\ldots,S_n)$.
We refer to \cite{FN2016} for the precise definitions of the relative \emph{free Fisher information} $\Phi^*(F|S)$, the relative \emph{free entropy} $\chi^*(F|S)$ and the relative \emph{free Stein discrepancy} $\Sigma^*(F|S)$ in the multivariate setting. Note the following difference of notation:  in \cite{FN2016}, $\Phi^*(F|S)$ is denoted by $\Phi^*(F|V_1)$, $\chi^*(F|S)$ corresponds to $\chi^*(F|V_1)-\chi^*(S|V_1)$ and $\Sigma^*(F|S)$ is denoted by $\Sigma^*(F|V_1)$.

Beyond the definitions, the results of the current section extend to $n$-tuples $F=(F_1,\ldots,F_n)$ of self-adjoint elements. Lemma~\ref{deBruijn} and Theorem~\ref{thHSI} are stated in the multivariate case in the original article~\cite{FN2016}. The free Talagrand inequality and  Lemma~\ref{lemmaWI} are also true in the multivariate setting, as proved in~\cite[Theorem 26 and its proof]{D2010}. Furthermore, it is straightforward to adapt the proof of Lemma~\ref{lemmaSigma}, Lemma~\ref{lemmaIS}, Proposition~\ref{propWS} and Proposition~\ref{freeWSH}.

Finally, the free WS inequality and the free WSH inequality are true: for all $n$-tuples $F=(F_1,\ldots,F_n)$ of self-adjoint elements of $(\mathcal{A},\tau)$, we have
$$ W_2(F,S) \leq \Sigma^*(F|S)\ \ \ \text{and}\ \ \ W_2(F,S) \leq \Sigma^*(F|S)\arccos\left(e^{-\frac{\chi^*(F|S)}{\Sigma^*(F|S)^2}}\right),$$
where $S$ is a $n$-tuple of free standard semicircular variables and $W_2$ is the free $2$-Wasserstein distance of \cite{BV2001}.
\section{Stein discrepancy in the Wigner chaos}\label{Sec:Wigner}
Let $\mathcal{P}_2(n)$ be the set of pairings of $\{1,\ldots,n\}$. Let $\pi$ be a pairing of $\{1,\ldots,n\}$. A quadruplet $1\leq i< j < k < l \leq n$ is called a crossing of $\pi$ if $\{i,k\}\in \pi$ and $\{j,l\}\in \pi$. We denote by $\mathcal{NC}_2(n)$ the set of pairings of $\{1,\ldots,n\}$ which have no crossings. A subset $\mathcal{S}$ of self-adjoint elements of $\mathcal{A}$ is said to be \emph{centred jointly semicircular} if, for all $S_1,\ldots,S_n\in \mathcal{S}$, we have
\begin{equation*}
\tau(S_1\cdots S_n)=\sum_{\pi\in \mathcal{NC}_2(n)}\prod_{\{i,j\}\in \pi}\tau(S_i S_j).\end{equation*}
Note that this definition is the analogue of the Wick formula for a family $\mathcal{N}$ of centred jointly Gaussian random variables: for all $N_1,\ldots,N_n\in \mathcal{N}$, we have
\begin{equation*}
\mathbb{E}(N_1\cdots N_n)=\sum_{\pi\in \mathcal{P}_2(n)}\prod_{\{i,j\}\in \pi}\mathbb{E}(N_i N_j).\end{equation*}
Furthermore, every element of a family of centred jointly semicircular variables is a standard semicircular variable up to a multiplicative constant.

In order to study families of centred jointly semicircular variables, we will use free stochastic analysis, and for this reason, in the three next sections, we introduce the general framework of Wigner-It\^o integrals and the very useful tool of free Malliavin derivative.

\subsection{The semicircular field over $L^2_{\mathbb{R}}(\mathbb{R}_+)$}
Let $\mathcal{S}$ be a family of centred jointly semicircular variables in some tracial $W^*$-probability space $(\mathcal{A},\tau)$. For any integer $n\geq 0$, we denote by $\mathcal{P}_n$ the \emph{Wigner chaos} of order $n$, that is to say the Hilbert space in $L^2(\mathcal{A},\tau)$ generated by the set
$\{1_{\mathcal{A}}\}\cup\{S_1\cdots S_k:1\leq k\leq n,S_1,\ldots, S_k \in \mathcal{S} \}.$
We denote by $\mathcal{H}_n$ the \emph{homogeneous Wigner chaos} of order $n$, defined by
$$\mathcal{H}_n:=\mathcal{P}_n\ominus \mathcal{P}_{n-1}\ \ (=\mathcal{P}_n\cap \mathcal{P}_{n-1}^{\perp}),$$
and by $\pi_n$ the orthogonal projection of $L^2(\mathcal{A},\tau)$ onto $\mathcal{H}_n$. It follows that
$\mathcal{P}_n=\bigoplus_{k=0}^n\mathcal{H}_k.$
As examples, $\mathcal{H}_0$ is the Hilbert space $\mathbb{C}\cdot 1_\mathcal{A}$, and $\mathcal{H}_1$ is the Hilbert space in $L^2(\mathcal{A},\tau)$ generated by $\mathcal{S}$.

From now, we will assume that $\mathcal{H}_1$ is separable.  Moreover, enlarging $\mathcal{S}$ (and $\mathcal{A}$) if necessary, we will assume that $\mathcal{H}_1$ is an infinite separable Hilbert space and that $\mathcal{S}$ is the set of self-adjoint elements of $\mathcal{H}_1$. In this case, $\mathcal{S}$ is an infinite separable real Hilbert space. Every infinite separable real Hilbert space is isometrically isomorphic to $L^2_{\mathbb{R}}(\mathbb{R}_+)$, the space of real-valued measurable functions on $\mathbb{R}_+$ which are square-integrable. As a consequence, $\mathcal{S}$ is a \emph{semicircular field over $L^2_{\mathbb{R}}(\mathbb{R}_+)$}, in the sense that $\mathcal{S}$ is a centred jointly semicircular family which can be encoded as $\mathcal{S}=\{S(h):h\in L^2_{\mathbb{R}}(\mathbb{R}_+)\}$, where $h\mapsto S(h)$ is an isomorphism of Hilbert space  from $L^2_{\mathbb{R}}(\mathbb{R}_+)$ to $\mathcal{S}$. This convention will make life easier in the future, and has the advantage to stay in the framework of \cite{BS1998,KNPS2012,M2015}.

\subsection{The Wigner chaos}Our references for this section and the next one are \cite{BS1998,KNPS2012,M2015}. The fact that $\mathcal{S}$ is a semicircular field over $L^2_{\mathbb{R}}(\mathbb{R}_+)$ allows us to encode the homogeneous Wigner chaoses $\mathcal{H}_n$ by the spaces $L^2(\mathbb{R}_+^n)$, the spaces of complex-valued measurable functions on $\mathbb{R}_+^n$ which are square-integrable.

More precisely, we denote by $H=L^2(\mathbb{R}_+)$ the complexified of $L^2_{\mathbb{R}}(\mathbb{R}_+)$. We set $H^{\otimes 0}=\mathbb{C}$.
The mapping$$I_n:h_1\otimes \cdots \otimes h_n\mapsto\pi_n(S(h_1)\cdots S(h_n))$$ for $h_1,\ldots,h_n\in L^2_{\mathbb{R}}(\mathbb{R}_+)$ can be uniquely extended to a linear isometry between $H^{\otimes n}$ and $\mathcal{H}_n$, or between $L^2(\mathbb{R}_+^n)$  (which is a concrete realization of $H^{\otimes n}$) and $\mathcal{H}_n$. This map $I_n:L^2(\mathbb{R}_+^n)\to \mathcal{H}_n$ is called the \emph{Wick map}, or the multiple \emph{Wigner-It\^o integral} (with the convention that $I_0(1)=1_{\mathcal{A}}$). A priori, a multiple Wigner-It\^o integral is not supposed to belongs to $\mathcal{A}$, but, thanks to the following version of Haagerup's inequality, it does.
\begin{theorem}[Theorem 5.3.4 of \cite{BS1998}]\label{Haagerup}For all $A\in \mathcal{H}_n$, $A$ is in $\mathcal{A}$ and we have
 $$\|A\|_{\mathcal{A}} \leq (n+1)\|A\|_{L^2(\mathcal{A},\tau)}.$$
\end{theorem}
\begin{corollary}Let $n\geq 0$. Then, $\mathcal{P}_n\subset \mathcal{A}$ holds. Furthermore, if $A,B\in \mathcal{P}_n$, then $AB\in \mathcal{P}_{2n}$ and 
$$\|AB\|_{L^2(\mathcal{A},\tau)}\leq (n+1)^2\|A\|_{L^2(\mathcal{A},\tau)}\|B\|_{L^2(\mathcal{A},\tau)}.$$
\end{corollary}
\begin{proof}Let us write $A=\sum_{k=0}^nA_k$, where each $A_k\in \mathcal{H}_k$. By Theorem~\ref{Haagerup}, we have
$$\|A\|_{\mathcal{A}} \leq \sum_{k=0}^n\|A_k\|_{\mathcal{A}}\leq (n+1)\sum_{k=0}^n\|A_k\|_{L^2(\mathcal{A},\tau)}=(n+1)\|A\|_{L^2(\mathcal{A},\tau)}.$$
Finally,
$\|AB\|_{L^2(\mathcal{A},\tau)}\leq\|AB\|_{\mathcal{A}}\leq \|A\|_{\mathcal{A}}\|B\|_{\mathcal{A}}\leq (n+1)^2\|A\|_{L^2(\mathcal{A},\tau)}\|B\|_{L^2(\mathcal{A},\tau)}$. Moreover, if $A,B\in\{1_{\mathcal{A}}\}\cup\{S_1\cdots S_k:1\leq k\leq n,S_1,\ldots, S_k \in \mathcal{S} \}$, then $AB\in \mathcal{P}_{2n}$, and we conclude by linearity and continuity.
\end{proof}
There is no reason that the product should behave well with the Wick map $I_n$. It is possible to get rid of this difficulty thanks to the following notion of contraction. Given $f\in L^2(\mathbb{R}_+^n)$ and $g\in L^2(\mathbb{R}_+^m)$, for every $0\leq p \leq n \wedge m$, the \emph{contraction} of $f$ and $g$ of order $p$ is the element of $L^2(\mathbb{R}_+^{n+m-2p})$ defined almost everywhere by
\begin{multline*}
f\overset{p}{\frown} g(t_1,\ldots,t_{n+m-2p}) \\=\int_{\mathbb{R}_+^{p}}f(t_1,\ldots,t_{n-p},s_p,\ldots,s_1)g(s_1,\ldots,s_{p},t_{n-p+1},\ldots,t_{n+m-2p})\diff s_1 \cdots \diff s_p.
\end{multline*}
\begin{prop}[Proposition 5.3.3 of \cite{BS1998}]\label{productofI}For all $f\in L^2(\mathbb{R}_+^n)$ and $g\in L^2(\mathbb{R}_+^m)$,
 $$I_n(f)I_m(g)=\sum_{p=0}^{n \wedge m}I_{n+m-2p}(f\overset{p}{\frown} g).$$
 \end{prop}
In particular,
 $\tau(I_n(f)I_m(g))=\delta_{n=m}\cdot (f\overset{n}{\frown} g).$ Given $f\in L^2(\mathbb{R}_+^n)$, the \emph{adjoint} of $f$ is the function $f^*(t_1,\ldots,t_n)=\overline{f(t_n,\ldots,t_1)}$ in $L^2(\mathbb{R}_+^n)$, in such a way that $I_n(f)^*=I_n(f^*)$. In particular, $I_n(f)$ is self-adjoint if and only if $f=f^*$. The notion of contraction is useful to compute the norm of Wigner-Itô integrals.
  \begin{prop}[Proof of Theorem 1.6 of \cite{KNPS2012}]\label{propfourth}Let $f\in L^2(\mathbb{R}_+^n)$ such that $\|f\|_{L^2(\mathbb{R}_+^{n})}=1$. For all $1\leq p \leq n-1$, we have
  $$\tau(|I_n(f)|^4)=2+\sum_{p=1}^{n-1}\|f \overset{p}{\frown} f^*\|_{L^2(\mathbb{R}_+^{2n-2p})}^2.$$
 \end{prop}
 It is also possible to define more sophisticated contractions. Let $n_1,\ldots,n_r$ be positive integers. We denote by $\mathcal{P}_2(n_1\otimes \cdots \otimes n_r)$ the set of pairings $\pi$ of $\mathcal{P}_2(n_1+\cdots+n_r)$ such that no block of $\pi$ contains more than one element from each interval set
$$\{1,\ldots,n_1\},\{n_1+1,\ldots,n_1+n_2\},\ldots,\{n_1+\cdots+n_{r-1}+1,\ldots,n_1+\cdots+n_r\}.$$
For all $f_1\in L^2(\mathbb{R}_+^{n_1}),\ldots,f_r\in L^2(\mathbb{R}_+^{n_r})$, we define the pairing integral of $f_1\otimes\cdots \otimes f_r$ with respect to $\pi \in \mathcal{P}_2(n_1\otimes \cdots \otimes n_r)$ to be the constant
$$\int_\pi f_1\otimes\cdots \otimes f_r:=\idotsint\limits_{\mathbb{R}_+^{(n_1+\cdots+n_r)/2}}(f_1\otimes\cdots \otimes f_r)(t_1,t_2,\ldots,t_{n_1+\cdots+n_r})\prod_{\{i,j\}\in \pi}\diff t_i,$$
where $t_i$ and $t_j$ are identified whenever $\{i,j\}\in \pi$.
 \begin{prop}[Lemma 2.1 of \cite{KNPS2012}]\label{propJanson}Let $n_1,\ldots,n_r>0$, $f_1\in L^2(\mathbb{R}_+^{n_1}),\ldots,f_r\in L^2(\mathbb{R}_+^{n_r})$ and $\pi \in \mathcal{P}_2(n_1\otimes \cdots \otimes n_r)$. Then,
 
 $$\left|\int_\pi f_1\otimes\cdots \otimes f_r\right|\leq \|f_1\|_{L^2(\mathbb{R}_+^{n_1})}\cdots \|f_r\|_{L^2(\mathbb{R}_+^{n_r})}.$$
 \end{prop}

\subsection{The free Malliavin derivative}
The (free) \emph{Malliavin derivative} is the unique unbounded operator
$$\begin{array}{crcl}
\nabla: & L^2(\mathcal{S},\tau) & \to & L^2(\mathbb{R}_+;L^2(\mathcal{S},\tau)\otimes L^2(\mathcal{S},\tau)) \\ 
 & A & \mapsto & \nabla A=(\nabla_t A)_{t\geq 0} \\ 
\end{array}  $$
such that, for all $h\in L^2_{\mathbb{R}}(\mathbb{R}_+)$, we have
$\nabla(S(h))=h\cdot 1 \otimes 1$ and such that $\nabla$ satisfies the following derivation rule: for all $A$ and $B$ in the algebra generated by $\mathcal{S}$, we have
$\nabla(AB)=\nabla(A)\cdot B+A\cdot \nabla(B)$, where the left and right actions of $\mathcal{S}$ are given by multiplication on the left leg and by opposite multiplication on the right leg.
\begin{prop}[Proposition 5.3.10 of \cite{BS1998}]\label{domainnabla}Let $n\geq 0$. The domain of $\nabla$ contains $\mathcal{P}_n$, and the restriction of $\nabla$ to $\mathcal{P}_n$ is a bounded linear operator.
\end{prop}
It is possible to describe explicitely the action of $\nabla$ on $\mathcal{P}_n$. For all $n,m\geq 0$, $I_n$ and $I_m$ induce the linear map $I_n\otimes I_m$ from the Hilbert space tensor product $L^2(\mathbb{R}_+^{n})\otimes L^2(\mathbb{R}_+^{m})$ to $\mathcal{H}_n\otimes \mathcal{H}_m$. The concrete realization $L^2(\mathbb{R}_+^{n+m})$ of the tensor product $L^2(\mathbb{R}_+^{n})\otimes L^2(\mathbb{R}_+^{m})$ allows us to consider $I_n\otimes I_m$ as a map
$$I_n\otimes I_m:L^2(\mathbb{R}_+^{n+m})\to\mathcal{H}_n\otimes \mathcal{H}_m$$
defined on $L^2(\mathbb{R}_+^{n+m})$ such that
$$I_n\otimes I_m:h_1\otimes \cdots \otimes h_{n+m}\mapsto\pi_n(S(h_1)\cdots S(h_n))\otimes \pi_n(S(h_{n+1})\cdots S(h_{n+m}))$$ for $h_1,\ldots,h_{n+m}\in L^2_{\mathbb{R}}(\mathbb{R}_+)$.

Finally, for all $f\in L^2(\mathbb{R}_+^{n})$, and for almost all $t\in \mathbb{R}_+$, we define  $f^k_t\in L^2(\mathbb{R}_+^{n-1})$ by the following equality, true for almost all $t\geq 0$:
$$f(t_1,\ldots,t_{k-1},t,t_{k+1},\ldots,t_n)=f^k_t(t_1,\ldots,t_{k-1},t_{k+1},\ldots,t_n).$$
\begin{prop}[Proposition 5.3.9 of \cite{BS1998}]\label{nabla}Let $n\geq 0$. The Malliavin derivative $\nabla$ maps $\mathcal{P}_n$ into $L^2(\mathbb{R}_+;\mathcal{P}_n\otimes \mathcal{P}_n)$. More precisely, if $f\in L^2(\mathbb{R}_+^{n})$, then, for almost all $t\geq 0$,
$$\nabla_t(I_n(f))=\sum_{k=1}^n (I_{k-1}\otimes I_{n-k})(f^k_t).$$
\end{prop}

Let $n\geq 0$. If $A_1\otimes  B_1\in \mathcal{P}_n\otimes \mathcal{P}_n$ and $A_2\otimes B_2\in \mathcal{P}_n \otimes \mathcal{P}_n$, we set
$$(A_1\otimes B_1)^*:=A_1^*\otimes B_1^*\text{ and }(A_1\otimes B_1)\sharp (A_2\otimes  B_2):=A_1A_2\otimes B_2B_1.$$
By anti-linearity and continuity, 
the mapping $A\mapsto A^*$ extends to an antilinear map
from $\mathcal{P}_n \otimes \mathcal{P}_n$ to itself. By bilinearity and continuity, we extend 
the mapping $(A, B)\mapsto A\sharp B$ to a continuous bilinear map
from $(\mathcal{P}_n \otimes \mathcal{P}_n)^2$ to $\mathcal{P}_{2n}\otimes \mathcal{P}_{2n}$. Indeed, if $A$ and $B$ are finite sums $A=\sum_k A_k\otimes C_k$ and $B=\sum_k B_k\otimes C_k$, with $\{C_k\}$ an orthonormal family, we have
\begin{multline*}\hspace{-0.4cm}\|A\sharp B\|_{L^2(\mathcal{A},\tau)\otimes L^2(\mathcal{A},\tau)}=\|\sum_{i,j}A_iB_j\otimes C_jC_i\|_{L^2(\mathcal{A},\tau)\otimes L^2(\mathcal{A},\tau)}\leq \sum_{i,j}\|A_iB_j\|_{L^2(\mathcal{A},\tau)}\| C_jC_i\|_{L^2(\mathcal{A},\tau)}\\
\leq \sum_{i,j}(n+1)^4\|A_i\|_{L^2(\mathcal{A},\tau)}\|B_j\|_{L^2(\mathcal{A},\tau)}\| C_j\|_{L^2(\mathcal{A},\tau)}\|C_i\|_{L^2(\mathcal{A},\tau)}=(n+1)^4\|A\|_{L^2(\mathcal{A},\tau)}\| B\|_{L^2(\mathcal{A},\tau)}
.\end{multline*}
By construction, we have $(A\sharp B)^*=B^*\sharp A^*$, and the map
$(A,B)\mapsto A\sharp B$ is associative on $\cup_n (\mathcal{P}_n \otimes \mathcal{P}_n)^2$. Moreover, the relation $$\langle A,B \rangle_{L^2(\mathcal{A},\tau)\otimes L^2(\mathcal{A},\tau)}=\tau \otimes \tau (B\sharp A^*),$$ which is true for $A=A_1\otimes  B_1\in \mathcal{P}_n\otimes \mathcal{P}_n$ and $B=A_2\otimes B_2\in \mathcal{P}_n \otimes \mathcal{P}_n$, extends by (anti)linearity and continuity to all $A, B \in \mathcal{P}_n \otimes \mathcal{P}_n$.

From the derivation rule of $\nabla$, we deduce the following chain rule for $A$ in the algebra generated by $\mathcal{S}$, and by density, for $A\in \mathcal{P}_n$.
\begin{prop}\label{chainrule}Let $A\in \mathcal{P}_n$ and $P\in \mathbb{C}[X]$. For almost all $t\geq 0$,
$\nabla_t(P(A))=(\partial P(A))\sharp (\nabla_t A).$
\end{prop}

\subsection{A free Stein kernel for elements of the Wigner chaos}
Let $n\geq 0$. The bounded linear maps $\tau\otimes \id:\mathcal{P}_n\otimes \mathcal{P}_n\to \mathcal{P}_n$ and $\id\otimes \tau:\mathcal{P}_n\otimes \mathcal{P}_n\to \mathcal{P}_n$ are defined by $\tau\otimes \id(A\otimes B)=\tau(A)\cdot B\text{\ and\ }\id\otimes \tau(A\otimes B)=\tau(B)\cdot A$ for all $A,B\in \mathcal{P}_n$.

\begin{lemma}\label{lemmanabla}For all $A,B\in \mathcal{P}_n$ such that $\tau(A)=0$ or $\tau(B)=0$, we have
$$\tau(AB)=\tau\left(\int_{\mathbb{R}_+}(\id\otimes \tau(\nabla_t A))\cdot (\tau\otimes \id(\nabla_t B))\diff t\right).$$
\end{lemma}
\begin{proof}If $\tau(A)=0$, we can assume that $\tau(A)=\tau(B)=0$ by replacing $B$ by $B-\tau(B)1_{\mathcal{A}}$ if necessary: indeed, $\tau(A\tau(B))=0$ and $\nabla_t(\tau(B)1_{\mathcal{A}})=0$. For the same reason, if $\tau(B)=0$, we can assume that $\tau(A)=\tau(B)=0$.

By linearity, it suffices to prove the result when $A=I_n(f)\in \mathcal{H}_n$ and $B=I_m(g)\in \mathcal{H}_m$ with $n,m>0$. By linearity and continuity of both side, we can assume that $f=f_1\otimes \cdots \otimes f_n$ and $g=g_1\otimes \cdots \otimes f_m$. We have
$$\nabla_t A=\sum_{i=1}^nf_i(t)\cdot  I_{i-1}(f_1\otimes \cdots \otimes f_{i-1})\otimes  I_{n-i}(f_{i+1}\otimes \cdots \otimes f_{n}))$$
and applying $\id\otimes \tau$ gives us $\id\otimes \tau(\nabla_t A)=f_n(t)I_{n-1}(f_1\otimes \cdots \otimes f_{n-1})$. Similarly, $\tau\otimes \id(\nabla_t B)=g_1(t)I_{m-1}(g_2\otimes \cdots \otimes g_{m-1})$.
Finally,
\begin{align*}
&\hspace{-1cm}\tau\left(\int_{\mathbb{R}_+}(\id\otimes \tau(\nabla_t A))\cdot (\tau\otimes \id(\nabla_t B))\diff t\right)\\
&=\int_{\mathbb{R_+}}f_n(t)g_m(t)\diff t \cdot \tau(I_{n-1}(f_1\otimes \cdots \otimes f_{n-1})I_{m-1}(g_2\otimes \cdots \otimes g_{m-1}))\\
&=\int_{\mathbb{R_+}}f_n(t)g_m(t)\diff t \cdot \delta_{n-1=m-1}(f_1\otimes \cdots \otimes f_{n-1})\overset{n}{\frown} (g_2\otimes \cdots \otimes g_{n-1})\\
&=\delta_{n=m}f\overset{n}{\frown} g\\
& =\tau(AB).\qedhere
 \end{align*}
\end{proof}
Recall that, if $A \in \mathcal{P}_n$ and $B \in \mathcal{P}_n\otimes\mathcal{P}_n$, the left action $A\cdot B$ is the multiplication on the left leg:
$$A\cdot B=(A\otimes 1_{\mathcal{A}}) \sharp B.$$
\begin{corollary}\label{corStein}For all $F\in \mathcal{P}_n$ such that $\tau(F)=0$, we have, for all polynomial $P\in \mathbb{C}[X]$,
$$\tau(P(F)F)=\tau\otimes \tau\left(\partial P(F)\sharp \int_{\mathbb{R}_+}\nabla_tF\sharp \left((\tau\otimes \id(\nabla_t F))\otimes 1_{\mathcal{A}}\right)\diff t\right).$$
In other words, if $F$ is self-adjoint,
$$\int_{\mathbb{R}_+}(\id\otimes \tau(\nabla_tF))\cdot (\nabla_tF)^*\diff t$$
is a Stein kernel of $F$ which belongs to $\mathcal{P}_{2n}$.
\end{corollary}
Before the demonstration, let us remark that the use of the divergence operator $\delta$ (the adjoint of the Malliavin derivative, which is not used in this article) gives an alternative proof of the above corollary. Indeed, one can compute $\delta ((\id\otimes \tau(\nabla F))\otimes 1_\mathcal{A})=F$, from which we deduce the following computation: 
\begin{multline*}
\tau(P(F)F^*)= \tau( P(F)(\delta ((\id\otimes \tau(\nabla F))))^*)=\int_{\mathbb{R}}\tau\otimes \tau( \nabla_t(P(F))\sharp (\id\otimes \tau(\nabla_tF))\otimes 1_\mathcal{A})^*)\diff s\\=\tau\otimes \tau\left(\partial P(F)\sharp \int_{\mathbb{R}_+}\nabla_tF\sharp \left((\tau\otimes \id(\nabla_t F))\otimes 1_{\mathcal{A}}\right)\diff t\right).
\end{multline*}
The same kind of computation allows Kemp, Nourdin, Peccati and Speicher to find the following Stein kernel (see \cite[Equation (4.16)]{KNPS2012}): for all $F=F^*\in \mathcal{H}_n$ such that $\tau(F)=0$,
$$\frac{1}{n}\int_{\mathbb{R}_+}(\nabla_tF)\cdot (\nabla_tF)^*\diff t$$
is a Stein kernel of $F$. However, as explained in the next section, it is more convenient to replace it by
$$\int_{\mathbb{R}_+}(\id\otimes \tau(\nabla_tF))\cdot (\nabla_tF)^*\diff t$$
in order to control the Stein discrepancy by the fourth moment.
\begin{proof}If $A\otimes B\in\mathcal{P}_n\otimes \mathcal{P}_n$ and $C\in \mathcal{P}_n$, we have $\tau((\id \otimes \tau(A\otimes B))\cdot C)=\tau(AC)\tau(B)=\tau\otimes \tau((A \otimes B)\sharp (C\otimes 1_{\mathcal{A}})).$
As a consequence, $\tau((\id \otimes \tau(D))\cdot C)=\tau\otimes \tau(D\sharp (C\otimes 1_{\mathcal{A}}))$ for all $D\in \mathcal{P}_n\otimes \mathcal{P}_n$ and $B\in \mathcal{P}_n$. Applied to Lemma~\ref{lemmanabla}, it gives us
$$\tau(P(F)F)= \tau\otimes \tau\left(\int_{\mathbb{R}_+}\nabla_t(P(F))\sharp \left((\tau\otimes \id(\nabla_t F))\otimes 1_{\mathcal{A}}\right)\diff t\right).$$
The first equality of the corollary is then a consequence of the chain rule of Proposition~\ref{chainrule}:  $\nabla_t(P(F))=\partial P(F)\sharp \nabla_t(F)$.

Because $\langle A^*,B \rangle_{L^2(\mathcal{A},\tau)\otimes L^2(\mathcal{A},\tau)}=\tau \otimes \tau (B\sharp A)$ for all $A,B\in \cup_n \mathcal{P}_n\otimes \mathcal{P}_n$, we know that the adjoint
$\int_{\mathbb{R}_+}\left((\tau\otimes \id(\nabla_t F))\otimes 1_{\mathcal{A}}\right)^*\sharp (\nabla_tF)^*\diff t$
of $\int_{\mathbb{R}_+}\nabla_tF\sharp \left((\tau\otimes \id(\nabla_t F))\otimes 1_{\mathcal{A}}\right)\diff t$
is a Stein kernel of $F$. We conclude by remarking that
$$\left((\tau\otimes \id(\nabla_t F))\otimes 1_{\mathcal{A}}\right)^*=(\id\otimes \tau)(\nabla_tF)\otimes 1_{\mathcal{A}},$$
where we used that $(\tau\otimes \id(\nabla_t F))^*=\id\otimes \tau(\nabla_t (F^*))=\id\otimes \tau(\nabla_t F)$ as shown in \cite[Lemma 4.6]{M2015}.
\end{proof}
\subsection{The Stein discrepancy and the fourth moment}

\begin{prop}\label{propFourthmoment}Let $n\geq 2$, and let $F$ be a self-adjoint element in the  homogeneous Wigner chaos $\mathcal{H}_n$ such that $\tau(F^2)=1$. Then,
$$\left\|\int_{\mathbb{R}_+}(\id\otimes \tau(\nabla_tF))\cdot (\nabla_tF)^*\diff t-1_{\mathcal{A}}\otimes 1_{\mathcal{A}}\right\|^2_{L^2(\mathcal{A},\tau)\otimes L^2(\mathcal{A},\tau)}\leq n^{3/2} \left(\tau(F^4)-2\right)^{1/2},$$
and in particular, we have the following bound on the Stein discrepancy:
$$\Sigma^*(F|S) \leq n^{3/4}\left(\tau(F^4)-2\right)^{1/4}.$$
\end{prop}
This proposition has to be compared with \cite[Theorem 3.7]{BC2017}, which states that there exists a certain constant $C_n$ such that
$$\left\|\frac{1}{n}\int_{\mathbb{R}_+}(\nabla_tF)\cdot (\nabla_tF)^*\diff t-1_{\mathcal{A}}\otimes 1_{\mathcal{A}}\right\|^2_{L^2(\mathcal{A},\tau)\otimes L^2(\mathcal{A},\tau)}\leq C_n \left(\tau(F^4)-2\right)$$
for $F$ in a particular subset of the self-adjoint elements in $\mathcal{H}_n$ (the fully symmetric elements). Unfortunately, as explained in \cite[Remark 3.9]{BC2017}, this inequality fails to be true for all self-adjoint element in $\mathcal{H}_n$. This explains the necessity of replacing the Stein kernel {$\frac{1}{n}\int_{\mathbb{R}_+}(\nabla_tF)\cdot (\nabla_tF)^*\diff t$} by $\int_{\mathbb{R}_+}(\id\otimes \tau(\nabla_tF))\cdot (\nabla_tF)^*\diff t$.
\begin{proof}Thanks to Proposition~\ref{corStein},
$$\int_{\mathbb{R}_+}(\id\otimes \tau(\nabla_tF))\cdot (\nabla_tF)^*\diff t $$ 
is a Stein kernel, and it suffices to prove the first inequality.

Let us write $F=I_n(f)$ for some $f=f^*\in L^2(\mathbb{R}_+^n)$. We define 
$\tilde{f}^k_t\in L^2(\mathbb{R}_+^{n-1})$  (for almost all $t\in \mathbb{R}_+$) by
$$\overline{f(t_1,\ldots,t_{k-1},t,t_{k+1},\ldots,t_n)}=\tilde{f}^k_t(t_{k-1},t_{k-2},\ldots,t_1,t_n,t_{n-1},\ldots,t_{k+1}).$$
Remark that $([I_{k-1}\otimes I_{n-k}](f^k_t))^*=[I_{k-1}\otimes I_{n-k}](\tilde{f}^k_t)$ holds for almost all $t\geq 0$ if $f=f_1\otimes \cdots \otimes f_n$, and by linearity and density, it holds for all $f\in L^2(\mathbb{R}_+^n)$.

Let us compute
\begin{align*}
(\id\otimes \tau)(\nabla_tF)\cdot (\nabla_tF)^*&=\left((\id\otimes \tau)\left(\sum_{k=1}^n[I_{k-1}\otimes I_{n-k}](f^k_t)\right)\right)\cdot \left(\sum_{k=1}^n[I_{k-1}\otimes I_{n-k}](f^k_t)\right)^*
\\&=I_{n-1}(f^n_t)\cdot\left(\sum_{k=1}^n[I_{k-1}\otimes I_{n-k}](\tilde{f}^k_t)\right)\\
&=\sum_{k=1}^n\sum_{p=0}^{k-1}[I_{n-1+k-1-2p}\otimes I_{n-k}](f^n_t \overset{p}{\frown} \tilde{f}^k_t)
\end{align*}
for almost all $t\geq 0$,
where the equality $$I_{n-1}(f^n_t)\cdot[I_{k-1}\otimes I_{n-k}](\tilde{f}^k_t)=\sum_{p=0}^{k-1}[I_{n-1+k-1-2p}\otimes I_{n-k}](f^n_t \overset{p}{\frown} \tilde{f}^k_t)$$ is proved by density, starting from functions of type $f^n_t=h_1\otimes \cdots \otimes h_{n-1}$ and $\tilde{f}^k_t=g_1\otimes \cdots \otimes g_{n-1}$ thanks to Proposition~\ref{productofI}.

Remark that, when $k=n$ and $p=n-1$, $f^n_t \overset{p}{\frown} \tilde{f}^k_t$ is the constant
$$\int_{\mathbb{R}_+^{n-1}}|f(t_1,\ldots,t_{n-1},t)|^2\diff t_1 \cdots \diff t_{n-1}$$
and, when integrating the corresponding term $[I_{n-1+k-1-2p}\otimes I_{n-k}](f^n_t \overset{p}{\frown} \tilde{f}^k_t)=(f^n_t \overset{p}{\frown} \tilde{f}^k_t)\cdot 1_{\mathcal{A}}\otimes 1_{\mathcal{A}}$ over $t$, we get \smash{$\|f\|_{L^2(\mathbb{R}_+^{n})}^2\cdot 1_{\mathcal{A}}\otimes 1_{\mathcal{A}}=\tau(F^2)\cdot 1_{\mathcal{A}}\otimes 1_{\mathcal{A}}=1_{\mathcal{A}}\otimes 1_{\mathcal{A}}$}. Consequently, it will disappear in the following computation:
\begin{align*}
&\hspace{-1cm}\left\|\int_{\mathbb{R}_+}(\id\otimes \tau)(\nabla_tF)\cdot (\nabla_tF)^*\diff t-1_{\mathcal{A}}\otimes 1_{\mathcal{A}}\right\|^2_{L^2(\mathcal{A},\tau)\otimes L^2(\mathcal{A},\tau)}\\
&=\left\|\sum_{k=1}^{n-1}\sum_{p=0}^{k-1}\int_{\mathbb{R}_+}[I_{n-1+k-1-2p}\otimes I_{n-k}](f^n_t \overset{p}{\frown} \tilde{f}^k_t)\diff t\right.\\
&\hspace{3cm}\left.+\sum_{p=0}^{n-1}\int_{\mathbb{R}_+}[I_{2n-2p-2}\otimes I_{n-k}](f^n_t \overset{p}{\frown} \tilde{f}^n_t)\diff t-1_{\mathcal{A}}\otimes 1_{\mathcal{A}}\right\|^2_{L^2(\mathcal{A},\tau)\otimes L^2(\mathcal{A},\tau)}\\
&=\sum_{k=1}^{n-1}\sum_{p=0}^{k-1}\left\|\int_{\mathbb{R}_+}(f^n_t \overset{p}{\frown} \tilde{f}^k_t)\diff t\right\|^2_{L^2(\mathbb{R}_+^{2n-2p-2})}+\sum_{p=0}^{n-2}\left\|\int_{\mathbb{R}_+}(f^n_t \overset{p}{\frown} \tilde{f}^n_t)\diff t\right\|^2_{L^2(\mathbb{R}_+^{2n-2p-2})}.
\end{align*}
Now, fixing $k$ and $p$, let us compute
\begin{align*}&f^n_t \overset{p}{\frown} \tilde{f}^k_t(t_1,\ldots,t_{n-p-1},r_1,\ldots,r_{n-p-1})\\&=\int_{\mathbb{R}_+^{p}}f^n_t(t_1,\ldots,t_{n-p-1},s_p,\ldots,s_1)\tilde{f}^k_t(s_1,\ldots,s_p,r_1,\ldots,r_{n-p-1})\diff s_1\cdots \diff s_p\\
&=\int_{\mathbb{R}_+^{p}}f(t_1,\ldots,t_{n-p-1},s_p,\ldots,s_1,t)\overline{f(r_{k-p-1},\cdots,r_1,s_p,\ldots,s_1,t,r_{n-p-1},\ldots,r_{k-p})}\diff s_1\cdots \diff s_p
\end{align*}
from which we deduce that
\begin{align*}
&\hspace{-1.5cm}\left\|\int_{\mathbb{R}_+}(f^n_t \overset{p}{\frown} \tilde{f}^k_t)\diff t\right\|^2_{L^2(\mathbb{R}_+^{2n-2p-2})}\\
=\int_{\mathbb{R}_+^{2n}}& f(t_1,\ldots,t_{n-p-1},s_p,\ldots,s_1,t)\overline{f(r_{k-p-1},\cdots,r_1,s_p,\ldots,s_1,t,r_{n-p-1},\ldots,r_{k-p})}\\
&\cdot \overline{f(t_1,\ldots,t_{n-p-1},s'_p,\ldots,s'_1,t')}f(r_{k-p-1},\cdots,r_1,s'_p,\ldots,s'_1,t',r_{n-p-1},\ldots,r_{k-p})\\
&\hspace{4cm}\diff s_1\cdots \diff s_p\ \diff t\  \diff s'_1\cdots \diff s'_p\ \diff t'\  \diff t_1 \ldots \diff t_{n-p-1}\ \diff r_1 \ldots \diff r_{n-p-1}.
\end{align*}
Because $f=f^*$, we have
\begin{align*}
&\hspace{-1.5cm}\left\|\int_{\mathbb{R}_+}(f^n_t \overset{p}{\frown} \tilde{f}^k_t)\diff t\right\|^2_{L^2(\mathbb{R}_+^{2n-2p-2})}\\
=\int_{\mathbb{R}_+^{2n}}& f(t_1,\ldots,t_{n-p-1},s_p,\ldots,s_1,t)f(r_{k-p},\ldots,r_{n-p-1},t,s_1,\ldots,s_p,r_1,\ldots,r_{k-p-1})\\
&\cdot f(t',s'_1,\ldots,s'_p, t_{n-p-1},\ldots,t_1)f(r_{k-p-1},\cdots,r_1,s'_p,\ldots,s'_1,t',r_{n-p-1},\ldots,r_{k-p})\\
&\hspace{4cm}\diff s_1\cdots \diff s_p\ \diff t\  \diff s'_1\cdots \diff s'_p\ \diff t'\  \diff t_1 \ldots \diff t_{n-p-1}\ \diff r_1 \ldots \diff r_{n-p-1}.
\end{align*}
If we integrate over $\diff t_1 \ldots \diff t_{n-p-1}$, we get $f \overset{n-p-1}{\frown} f$.
Thus, we are left to
\smash{$\int_\pi  (f \overset{n-p-1}{\frown} f)\otimes f \otimes f $}
for some pairing $\pi\in \mathcal{P}_2((2p+2)\otimes n \otimes n)$. Using Proposition~\ref{propJanson} leads to the following bounds
\begin{align*}
&\hspace{-1.5cm}\left\|\int_{\mathbb{R}_+}(f^n_t \overset{p}{\frown} \tilde{f}^k_t)\diff t\right\|^2_{L^2(\mathbb{R}_+^{2n-2p-2})}\leq \left|\int_\pi  (f \overset{n-p-1}{\frown} f)\otimes f\otimes f\right|\\
&\leq \left\|f \overset{n-p-1}{\frown} f\right\|_{L^2(\mathbb{R}_+^{2p+2})}\cdot \left\|f\right\|_{L^2(\mathbb{R}_+^{n})}\cdot \left\|f\right\|_{L^2(\mathbb{R}_+^{n})}= \left\|f \overset{n-p-1}{\frown} f\right\|_{L^2(\mathbb{R}_+^{2p+2})}.
\end{align*}

The norm of the contraction which appears can be controlled by $\tau(F^4)-2$ in the following way: Proposition~\ref{propfourth} ensures that \smash{$\sum_{i=1}^{n-1}\|f \overset{i}{\frown} f\|_{_{L^2(\mathbb{R}_+^{2n-2i})}}^2$} is bounded by $\tau(I_n(f)^4)-2=\tau(F^4)-2$, and consequently,
$$ \sum_{i=1}^{n-1}\left\|f \overset{i}{\frown} f\right\|_{L^2(\mathbb{R}_+^{2n-i})}
\leq \sqrt{n-1}\left(\sum_{i=1}^{n-1}\left\|f \overset{i}{\frown} f\right\|_{L^2(\mathbb{R}_+^{2n-i})}^2\right)^{1/2}\leq\sqrt{n} \left(\tau(F^4)-2\right)^{1/2}
.$$

Finally, \begin{align*}
&\hspace{-1cm}\left\|\int_{\mathbb{R}_+}(\id\otimes \tau)(\nabla_tF)\cdot (\nabla_tF)^*\diff t-1\otimes 1\right\|^2_{L^2(\mathcal{A},\tau)\otimes L^2(\mathcal{A},\tau)}\\&=
\sum_{k=1}^{n-1}\sum_{p=0}^{k-1}\left\|\int_{\mathbb{R}_+}(f^n_t \overset{p}{\frown} \tilde{f}^k_t)\diff t\right\|^2_{L^2(\mathbb{R}_+^{2n-2p-2})}+\sum_{p=0}^{n-2}\left\|\int_{\mathbb{R}_+}(f^n_t \overset{p}{\frown} \tilde{f}^n_t)\diff t\right\|^2_{L^2(\mathbb{R}_+^{2n-2p-2})}\\
&\leq\sum_{k=1}^{n-1}\sum_{p=0}^{k-1}\left\|f \overset{n-p-1}{\frown} f\right\|_{L^2(\mathbb{R}_+^{2p+2})}+\sum_{p=0}^{n-2}\left\|f \overset{n-p-1}{\frown} f\right\|_{L^2(\mathbb{R}_+^{2p+2})}\\
&\leq\sum_{k=1}^{n-1}\sqrt{n} \left(\tau(F^4)-2\right)^{1/2}+\sqrt{n} \left(\tau(F^4)-2\right)^{1/2}\\
 &\leq n^{3/2} \left(\tau(F^4)-2\right)^{1/2}.\qedhere
\end{align*}
\end{proof}
\subsection{Fourth moment theorems in the Wigner chaos}\label{Fully}Given $n\geq 1$, $f\in L^2(\mathbb{R}_+^n)$ is called \emph{mirror symmetric} if $f=f^*$. An element $F$ of the homogeneous Wigner chaos $\mathcal{H}_n$ is self-adjoint if and only if $F=I_n(f)$ for a certain mirror symmetric $f\in L^2(\mathbb{R}_+^n)$, which allows us to state the free fourth moment theorem, due to Kemp, Nourdin, Peccati and Speicher, in the following form.

\begin{theorem}[Theorem 1.3 of \cite{KNPS2012}]Let $n\geq 2$, and let $(F_k)_{k\geq 1}$ be a sequence of self-adjoint element in the  homogeneous Wigner chaos $\mathcal{H}_n$, each with $\tau(F_k^2)=1$. The distribution of $F_k$ converges weakly to the distribution of a semi-circular variable $S$ if and only if $\tau(F_k^4)$ converges to $\tau(S^4)=2$ as $k$ tends to $\infty$.
\end{theorem}
A natural question is whether it is possible to quantify the convergence in the fourth moment theorem, as in the classical case. Kemp, Nourdin, Peccati and Speicher give an answer in the $2$-th chaos, using the following distance. For all self-adjoint random variables $A$ and $B$ in $(\mathcal{A},\tau)$, the distance $d_{\mathcal{C}_2}$, introduced in \cite{KNPS2012}, is defined by
$$d_{\mathcal{C}_2}(A,B):=\sup_{\substack{h\in \mathcal{C}_2\\ \mathcal{I}_2(h)\leq 1}}\left|\tau(h(A))-\tau(h(B))\right|,$$ where $\mathcal{C}_2$ is the set of Fourier transforms $h:\mathbb{R}\to \mathbb{C}$ of complex measures $\mu$ such that the quantity $\mathcal{I}_2(h):=\int_\mathbb{R}x^2|\mu|(\diff x)$ is finite.

\begin{prop}[Corollary 1.12 of \cite{KNPS2012}]Let $F$ be a self-adjoint element in the  homogeneous Wigner chaos $\mathcal{H}_2$, such that $\tau(F^2)=1$. Then,
$$d_{\mathcal{C}_2}(F,S)  \leq \frac{1}{2}\sqrt{\frac{3}{4}}\left(\tau(F^4)-2\right)^{1/2}.$$
\end{prop}
Given $n\geq 1$, $f\in L^2(\mathbb{R}_+^n)$ is called \emph{(fully) symmetric} if $f$ is real-valued and, for any permutation $\sigma$ of $\{1,\ldots,n\}$, $f=f_\sigma$ in $L^2(\mathbb{R}_+^n)$, where $f_\sigma(t_1,\ldots,t_n)=f(t_{\sigma(1)},\ldots,t_{\sigma(n)})$.  A fully symmetric function is mirror symmetric, however, the set of element $I_n(f)$ where $f$ is fully symmetric is a strict subset of the self-adjoint elements of $\mathcal{H}_n$. By restricting themselves to the case of fully symmetric kernels,  Bourguin and Campese prove the following quantitative bound.
\begin{prop}[Corollary 3.8 of \cite{BC2017}]Let $n\geq 2$, and let $F$ be a self-adjoint element in the  homogeneous Wigner chaos $\mathcal{H}_n$, which can be written $F=I_n(f)$ for a fully symmetric $f\in L^2(\mathbb{R}_+^n)$, and such that $\tau(F^2)=1$. Then,
$$d_{\mathcal{C}_2}(F,S)  \leq \sqrt{C_n}\left(\tau(F^4)-2\right)^{1/2}$$
for a constant $C_n$ which grows asymptotically linearly with $n$.
\end{prop}

In fact, the following theorem shows that it is possible to control the distance $d_{\mathcal{C}_2}(F,S)$, and even the Wasserstein distance $W_2(F,S)$ of every self-adjoint element of a homogeneous Wigner chaos.
\begin{theorem}\label{mainth}Let $n\geq 2$, and let $F$ be a self-adjoint element in the  homogeneous Wigner chaos $\mathcal{H}_n$ such that $\tau(F^2)=1$. Then,
$$d_{\mathcal{C}_2}(F,S)\leq W_2(F,S)  \leq \Sigma^*(F|S) \leq n^{3/4}\left(\tau(F^4)-2\right)^{1/4}.$$
\end{theorem}

\begin{proof}Thanks to Proposition~\ref{propWS} and Proposition~\ref{propFourthmoment}, it only remains to prove that $d_{\mathcal{C}_2}(F,S)\leq W_2(F,S)$. It is a consequence of Lemma~\ref{prop:W}.
\end{proof}

\begin{lemma}For all self-adjoint random variables $A$ and $B\in \mathcal{A}$ and all function $f$ twice differentiable, we have\label{prop:W}
$$|\tau(f(A))-\tau(f(B))|\leq |\tau(A)-\tau(B)|\cdot |f'(0)|+ \frac{1}{2}(\tau(A^2)^{1/2}+\tau(B^2)^{1/2})W_2(A,B)\cdot \|f''\|_\infty.$$
In particular, if $\tau(A)=\tau(B)=0$ and $\tau(A^2)=\tau(B^2)=1$, we have
$$d_{\mathcal{C}_2}(A,B)\leq W_2(A,B).$$
\end{lemma}
\begin{proof}Because the functions $h$ of $\mathcal{C}_2$ such that  $\mathcal{I}_2(h)\leq 1$ are twice differentiable with $\|h''\|_\infty\leq 1$, the second inequality is a consequence of the first inequality. Let us assume that $\|f''\|_\infty<+\infty$. We rewrite
\begin{align*}
f(x)-f(y)&=\int_y^xf'(t) \diff t\\
&=\int_y^x(f'(0)+\int_0^tf''(s)\diff s) \diff t\\
&=(x-y)f'(0)+\int_y^x\int_0^tf''(s)\ \diff s\ \diff t.
\end{align*}
For all $\pi \in \Pi(A,B)$, we can integrate the above equality and get
$$\tau(f(A))-\tau(f(B))=(\tau(A)-\tau(B))f'(0)+\int_{\mathbb{R}^2}\int_y^x\int_0^tf''(s)\ \diff s\ \diff t\ \diff \pi(x,y).$$
Taking the absolute value yields
$$|\tau(f(A))-\tau(f(B))|\leq |\tau(A)-\tau(B)|\cdot |f'(0)|+\|f''\|_\infty \int_{\mathbb{R}^2}\int_{x\wedge y}^{x \vee y} |t|\ \diff t\ \diff \pi(x,y).$$
Let us remark that $|x\wedge y|\cdot x \vee y-|x \vee y|\cdot x\wedge y\geq 0$, which allows to compute
\begin{align*}
\int_{x\wedge y}^{x \vee y} |t|\diff t=&\frac{1}{2}(|x\vee  y|\cdot x\vee  y-|x\wedge y|\cdot x\wedge y)\\
\leq& \frac{1}{2}(|x\vee  y|\cdot x \vee y-|x\wedge y|\cdot x\wedge y+|x\wedge y|\cdot x \vee y-|x \vee y|\cdot x\wedge y)\\
&=\frac{1}{2}(|x\vee  y|+|x\wedge  y|)(x\vee  y-x\wedge  y)\\
&=\frac{1}{2}(|x|+|y|)|x-  y|.
\end{align*}
Therefore, by using Hölder's inequality,
\begin{align*}
&\int_{\mathbb{R}^2}\int_{x\wedge y}^{x \vee y} |t|\ \diff t\ \diff \pi(x,y)\leq \frac{1}{2}\int_{\mathbb{R}^2}(|x|+|y|)\cdot|x-  y|\diff \pi(x,y)\\
&\leq \frac{1}{2}\left(\int_{\mathbb{R}^2}|x|^2\diff \pi(x,y)\int_{\mathbb{R}^2}|x-y|^2\diff \pi(x,y)\right)^{1/2}+\frac{1}{2}\left(\int_{\mathbb{R}^2}|y|^2\diff \pi(x,y)\int_{\mathbb{R}^2}|x-y|^2\diff \pi(x,y)\right)^{1/2}\\
&= \frac{1}{2}(\tau(A^2)^{1/2}+\tau(B^2)^{1/2})\left(\int_{\mathbb{R}^2}|x-y|^2\diff \pi(x,y)\right)^{1/2}.
\end{align*}
Finally, $|\tau(f(X))-\tau(f(Y))|$ is bounded by
$$ |\tau(A)-\tau(B)|\cdot |f'(0)|+\frac{1}{2}(\tau(A^2)^{1/2}+\tau(B^2)^{1/2})\left(\int_{\mathbb{R}^2}|x-y|^2\diff \pi(x,y)\right)^{1/2}\|f''\|_\infty.$$
and we conclude by minimizing over $\pi\in \Pi(A,B)$.
\end{proof}

\subsection*{Acknowledgements} The author wishes to thank Tobias Mai and Roland Speicher for plenty of useful discussions, and in particular Tobias Mai for introducing him to this problem and motivating him to establish Proposition~\ref{propWS}.

\bibliographystyle{amsalpha}
\bibliography{Stein_kernel_in_Wigner_chaos}

\end{document}